\documentclass[a4paper]{article}
\usepackage{amssymb}
\usepackage{amsmath,amsfonts,amsthm}
\usepackage{xcolor}
\usepackage{enumerate}

\newtheorem{theorem}{Theorem}[section]

\newtheorem{proposition}[theorem]{Proposition}

\newtheorem{conjecture}{Conjecture}

\usepackage{fullpage}
\allowdisplaybreaks

\begin{document}
	
	\numberwithin{equation}{section}
	
	\title{Topology at infinity of complete gradient Schouten solitons}
	
	\author{V. Borges\footnote{Universidade Federal do Pará, Faculdade de Matemática, 66075-110, Belém-PA, Brazil, valterborges@ufpa.br.} \quad H. A. Rosero-Garc\'ia \footnote{ Universidade de Bras\'{\i}lia,
			Department of Mathematics,
			70910-900, Bras\'{\i}lia-DF, Brazil, hector.garcia@unb.br, Partially supported by CAPES and CNPq.}
		\quad J. P. dos Santos \footnote{Universidade de Bras\'{\i}lia,
			Department of Mathematics,
			70910-900, Bras\'{\i}lia-DF, Brazil, joaopsantos@unb.br. Partially supported by CNPq 315614/2021-8.} 
	} 
	
	\date{}
	
	\maketitle{}
	
	
	
	\begin{abstract}
		In this paper, we study ends of complete gradient non-trivial Schouten solitons. Without any additional assumptions, we show the shrinking ones have finitely many ends, and the expanding ones are connected at infinity. We also provide information regarding the non-parabolicity of the ends in the shrinking setting, and on the behavior of the potential function and volume growth in the expanding case.
	\end{abstract}
	
	\vspace{0.2cm} 

	\section{Introduction and Main Results}

We say that a Riemannian manifold $(M^n,g)$ is a {\it gradient Schouten soliton} if there are a  function $f\in C^{\infty}(M)$ and a constant $\lambda\in\mathbb{R}$ satisfying
 \begin{equation}\label{schoutensoliton}
		\emph{Ric}+\nabla\nabla f = \left(\frac{R}{2(n-1)}+\lambda\right)g,
	\end{equation}
	where $\nabla\nabla f$, $Ric$ and $R$ represent the Hessian of $f$, the Ricci tensor and the scalar curvature of $g$, respectively. Gradient Schouten solitons constitute a particular class of a more general structure of solitons. Namely, given $\rho\in\mathbb{R}$, we say $(M^n,g,f,\lambda)$ is a {\it gradient} $\rho$-{\it Einstein soliton} if 
	\begin{align}\label{genfundeq}
		Ric+\nabla\nabla f=\left(\rho R+\lambda\right)g.
	\end{align}
	Therefore, Schouten solitons are $\rho$-Einstein solitons for $\rho=1/2(n-1)$. In any case, $f$ is called \emph{potential funcion} and $(M^n,g,f,\lambda)$ is said to be \textit{shrinking}, \textit{steady} or \textit{expanding}, if $\lambda>0$, $\lambda=0$ or $\lambda<0$, respectively.

    { Given $\rho\in\mathbb{R}$, a {\it Ricci-Bourguignon flow} ({\it associated to $\rho$}) on a manifold $M$ is a family of Riemannian metrics $g(t)$ defined on $M$ satisfying
		\begin{align}\label{rhoeinseq}
			\frac{\partial}{\partial t}g=-2(Ric-\rho R g).
		\end{align}
		Notice that when $\rho=0$, we have the Ricci flow. The principal symbol of the operator on the right hand side of \eqref{rhoeinseq} was computed in \cite{catino2}. This was used to show short time existence on compact manifolds, when $\rho<1/2(n-1)$. For $\rho>1/2(n-1)$, it has negative eigenvalues, indicating the problem is not likely to have solution in general. In the remaining case $\rho=1/2(n-1)$, in which $g(t)$ is called {\it Schouten flow}, the principal symbol has no negative eigenvalues but way too many zeros, and the classical theory does not apply, isolating $1/2(n-1)$ as a bifurcation value of \eqref{rhoeinseq}. It was shown in \cite{catino1} that gradient $\rho$-Einstein solitons generate solutions invariant by dilations and diffeomorphisms of \eqref{rhoeinseq}. This is one of the main motivations for studying such manifolds.}

        Another major motivation comes from the fact that \eqref{genfundeq} generalizes at the same time Einstein's and Ricci soliton's equations, both of which have widely been investigated in the last decades. It raises a natural examination of how analytic, geometric or topological properties of these structures change {as $\lambda$ and $\rho$ vary}. For instance, on the one hand, Einstein manifolds have constant scalar curvature and Schouten solitons have bounded scalar curvature \cite{brgs}. On the other hand, complete gradient Ricci solitons have scalar curvature bounded from below and little is known about upper bounds, while it is still unknown if $\rho$-Einstein solitons, in general, have bounds on the scalar curvature.

		Catino \textit{et al.} \cite{catino}  showed that complete gradient steady Schouten solitons with $\nabla f$ complete are Ricci flat. Still in \cite{catino}, rotationally symmetric non-trivial gradient steady $\rho$-Einstein solitons were constructed for $\rho<1/2(n-1)$ or $\rho\geq1/(n-1)$ using similar ideas on the construction of the Bryant soliton \cite{brnt}. It was also proved in \cite{catino} that every complete $3$-dimensional shrinking Schouten soliton is isometric to a finite quotient of either $\mathbb{S}^3$, $\mathbb{R}^3$ or $\mathbb{R}\times\mathbb{S}^2$. The same result was already known for Ricci solitons \cite{caochenzhou}. Whenever a $\rho$-Einstein soliton has constant potential function, we say it is {\it trivial}. Otherwise, it is said to be {\it non-trivial}. The rigid $\rho$-Einstein solitons constructed on $\mathbb{R}^{n-k}\times N^k$ for $k<n$ provide simple examples of non-trivial $\rho$-Einstein solitons (see \cite{brgs,catino} for more details).

         As observed in \cite{catino}, compact shrinking gradient Schouten solitons have also appeared in characterizing equality in the {\it de Lellis-Topping inequality}, also known as {\it Almost-Schur lemma}. For more details, see \cite{lelistop}. On the other hand, the locally conformally flat ones characterize the solitons of another geometric flow, known as {\it Riemann flow}. See \cite{tobabame} and references therein for details.
 
		In this paper, we are interested in the ends structure of a complete gradient Schouten soliton, that is, in its topology at infinity. In \cite{pLi0}, Li and Tam developed techniques to count the number of ends of Riemannian manifolds, which were later used to study complete manifolds with positive spectrum \cite{pLiW,pLiW2}. Munteanu and Wang \cite{MuWa1,MuWa2,MuWa4,MuWa5} have been using a variation of this technique to study the ends of Ricci solitons. In particular, they managed to show in \cite{MuWa1} that a gradient steady Ricci soliton must be connected at infinity. Such a result does not for $\lambda\neq0$ once $\mathbb{R}\times N^{n-1}$, for a compact manifold $N$ of constant Ricci curvature $\lambda$, is a non-trivial Ricci soliton with two ends. Nevertheless, it was proved in \cite{famazha} that gradient shrinking Ricci solitons with bounded scalar curvature have {\it finite topological type}, i.e., are homeomorphic to the interior of a compact manifold with boundary. In particular, these solitons have finitely many ends. In \cite{famazha} the authors proposed the following
	\begin{conjecture}[Fang, Man and Zhang]\label{conjFTT}
		Any shrinking Ricci soliton has finite topological type.
	\end{conjecture}
	
	While this conjecture is still unsolved, the corresponding issue for gradient shrinking Schouten solitons has an affirmative answer, as stated below.
	
	\begin{theorem}\label{thm1item1}
		Any non-trivial complete gradient shrinking Schouten soliton has finite topological type. In particular, it has finitely many ends, each of them homeomorphic to the product of a compact manifold with $[0,\infty)$.
		
	\end{theorem}
	
	The theorem above follows from the properties of the potential function of a Schouten soliton obtained in \cite{brgs}. See Section \ref{shrsch} for the proof.
	
	A natural question tied to both the conjecture and the theorem above is whether we can estimate the number of ends in terms of the geometry of $M$. It was shown in \cite{MuWa4} by Munteanu and Wang that Kähler gradient shrinking Ricci solitons are connected at infinity, i.e., they have only one end. In \cite{MuWa5}, they obtained the same conclusion in the Riemannian setting, assuming the scalar curvature is bounded from above by $\frac{2n}{3}\lambda$. In \cite{brstrs}, the authors were able to extend the last result to complete gradient shrinking $\rho$-Einstein solitons with $\rho\in[0,1/2(n-1))$. This interval is required to rule out the existence of $\varphi$-parabolic ends, and to show there is no more than one $\varphi$-non-parabolic end, where the weight function is $\varphi=-af$ with $a>0$. Nevertheless, the same argument seems to not apply to shrinking Schouten solitons, once an important inequality degenerates when $\rho=1/2(n-1)$. We notice that, as a consequence of the results of Derdzinski and Maschler \cite{dema,ma}, Kähler Schouten solitons have constant scalar curvature, and hence are Ricci solitons, leaving nothing to prove.
	
	An important step in the proof of the results described above is deciding whether a given end is parabolic with respect to some weighted measure or not. Here we are able to provide the following:
	
	\begin{theorem}\label{thm1}
		Let $(M^n,g,f,\lambda)$ be a complete non-trivial gradient shrinking Schouten soliton. Then:
		\begin{enumerate}[i)]
			\item  $M$ has at most one $f$-non-parabolic end.
			\item If
			\begin{align}\label{aperture}
				R\leq\alpha<\frac{2(n-1)(n-2)\lambda}{n},
			\end{align}
			then all ends of $M$  are non-parabolic.
		\end{enumerate}
	\end{theorem}
	To prove the theorem above we use Li-Tam theory \cite{endharm}. We notice that item {\it 2.} is a analogous to a result due to Munteanu and Sesum \cite{MuSe}, which asserts that all ends of shrinking Ricci soliton are non-parabolic, if the scalar curvature satisfies $R\leq\alpha<(n-2)\lambda$ for some constant $\alpha$. We observe that the upper bound \eqref{aperture} is optimal, in the sense that if $R=\frac{2}{n}(n-1)(n-2)\lambda$ the result no longer holds true. A simple counter-example is the shrinking Schouten soliton $\mathbb{R}^{2}\times\mathbb{S}^{n-2}$, which is parabolic.  
		
	As mentioned above, the growth of $f$ is an important tool in the investigation of gradient shrinking Schouten soliton. A similar result for the gradient expanding Schouten soliton is not true, as the potential function of the rigid soliton $\mathbb{R}^{n-k}\times\mathbb{H}^{k}$, $k\geq1$, is not proper. However, the following growth holds in this case
	
	\begin{theorem}\label{grwth}
		Let $(M^n,g,f,\lambda)$ be a complete non-trivial expanding Schouten soliton. Then
		\begin{align}\label{grtbhv}
			\lambda r^2-cr+f(p)\leq\inf_{x\in\partial B_{r}(p)}f(x)\leq\frac{\lambda}{4}r^2+cr^{\frac{3}{2}}\sqrt{\ln r}+f_{0},
		\end{align}
		for $r>0$ large enough, $c>0$ independent of $x$, and $f_{0}=\max_{M}f$. Furthermore, there is $C>0$ so that
		\begin{align}\label{volgrwt}
			Vol(B_{r}(p))\leq Ce^{\sqrt{-4\lambda(n-1)}r}.
		\end{align}
	\end{theorem}
	
	We observe that the volume growth of geodesic balls in the shrinking case is at least linear and at most polynomial \cite{brgs}. For the upper growth regarding the $f$-volume see Theorem 4.1 of \cite{weiwilie}.
	
	Now we describe our result on the topology at infinity of complete gradient expanding Schouten solitons. This is motivated by a result for expanding Ricci solitons, proved by Munteanu and Wang in \cite{MuWa2}, that we state below for the reader's convenience.
	
	\begin{theorem}[Theorem 1.4 of \cite{MuWa2}]\label{thRNcompact}
		Let $(M,g,f,\lambda)$ be a complete gradient expanding Ricci soliton. Assume that the scalar curvature satisfies $R\geq-(n-1)\lambda$, then either $M$ is connected at infinity or $M=\mathbb{R}\times N^{n-1}$, where $N$ is a compact Einstein manifold and $\mathbb{R}$ is the Gaussian expanding soliton. 
	\end{theorem}
	
	
	The limitation imposed on the scalar curvature in the theorem above is used in two main parts of the proof. First, they rule out the existence of $f$-parabolic ends. This is done following arguments of \cite{pLiW,pLiW2}, providing volume decay for $f$-parabolic ends, where identities proper to Ricci solitons are critical in the proof. Second, they show rigidity when more than one $f$-non-parabolic end is considered. Later on, the first part of the argument was simplified in \cite{wan2} using the growth of $f$. Using estimates obtained in \cite{brgs}, we are able to proceed as in \cite{wan2} to eliminate the existence of $f$-parabolic ends of gradient expanding Schouten soliton and to carry out the proof of the rigidity when more than one $f$-non-parabolic end is considered, as in \cite{MuWa5}. Thus we have the following satisfactory result.
	
	\begin{theorem}\label{thexpss}
		Let $(M,g,f,\lambda)$ be a complete non-trivial expanding Schouten soliton. Then $M$ is either connected at infinity or $M$ is isometric to the product $\mathbb{R}\times N^{n-1}$, where $N$ is a compact Einstein manifold and $\mathbb{R}$ is the Gaussian expanding soliton.
	\end{theorem}
	
	Compare Theorem \ref{thexpss} with {\it Corollary 1.1.} of \cite{MuWa1} and {\it question 1.1} of \cite{MuWa5}.	We would like to point out that, according to the authors of \cite{MuWa5}, it is not known if the assumption on the scalar curvature is superfluous in Theorem \ref{thRNcompact}. We observe that the scalar curvature of a complete gradient Schouten soliton always satisfies the bounds $2(n-1)\lambda\leq R\leq0$.
		
	This paper is organized as follows. In Section \ref{preliminariesresults}, we collect important preliminary results and definitions for our proofs, and it is divided into two subsections. The first subsection contains a few results concerning gradient Schouten solitons, including estimates for the scalar and potential function. We also present an improved estimate for $\vert\nabla f\vert^2$ in terms of any bounds on the scalar curvature. The second subsection collects the main definitions concerning Li-Tam theory and criteria for parabolicity used to prove our main results. In Section \ref{shrsch}, we prove our results concerning ends of gradient shrinking Schouten solitons, namely, Theorem \ref{thm1item1} and Theorem \ref{thm1}. In Section \ref{expsch} we prove the theorems on gradient expanding Schouten solitons, that is, Theorem \ref{grwth} and Theorem \ref{thRNcompact}. To obtain these results, we prove a weighted Poincaré inequality for the measure $e^{-f}dV$ and a lower bound to the bottom spectrum of $\Delta_{f}$. We also show the potential function decays at least quadratically.
	
	\section{Preliminary Results}\label{preliminariesresults}
	
	The goal of this section is twofold. In Subsection \ref{sub1}, we recall some results on Schouten solitons from \cite{brgs, catino, catino1}, and we present a gradient estimate for Schouten solitons that depend on bounds for its scalar curvature. In Subsection \ref{sub2}, we collect important definitions and results concerning parabolicity on smooth metric measure spaces found in \cite{MuWa1,MuWa2,MuWa4,MuWa5,endharm,wan1,wan2}.
	
	\subsection{Basics on Schouten Solitons}\label{sub1}
	
	The following identities obtained in \cite{catino, catino1} are essential in the study of Schouten solitons:
	
	\begin{proposition}[\cite{catino, catino1}]\label{thmcatino}
		If $(M^n,g,f,\lambda)$ is a gradient Schouten soliton, then
		\begin{equation*}
			\begin{split}
				\Delta f&=n\lambda-\frac{n-2}{2(n-1)}R,\\
				\mbox{\emph{Ric}}(\nabla f&,X)=0, \forall \,X\in\mathfrak{X}(M),\\
				\langle\nabla f,\nabla R\rangle+&\left(\frac{R}{n-1}+2\lambda\right)R=2\vert\mbox{\emph{Ric}}\vert^2.
			\end{split}
		\end{equation*}
	\end{proposition}
	
	These identities have been used by Catino and Mazzieri to establish rigidity results for these solitons. See Section 5 of \cite{catino}. These identities have also been used in the proof of the following estimates for $f$ and $R$ on a complete non-steady Schouten soliton.
	\begin{theorem}[\cite{brgs}]\label{th11val33}
		Let $(M^n,g,f,\lambda)$, $\lambda\neq0$, be a complete non-compact Schouten soliton with $f$ non-constant. If $\lambda>0$ $($respectively $\lambda<0$$)$, then the potential function $f$ attains a global minimum $($resp. maximum$)$ and is unbounded above $($resp. below$)$. Furthermore
		\begin{equation}\label{111valter33}
			0\leq \lambda R\leq 2(n-1)\lambda^2,
		\end{equation}
		\begin{equation}\label{112valter33}
			2\lambda(f-f_0)\leq \vert\nabla f\vert^2\leq 4\lambda(f-f_0),
		\end{equation}
		where $\displaystyle f_0=\min_Mf$ $($resp. $ \displaystyle f_0=\max_Mf$$)$.
	\end{theorem}
	
	The theorem above can be seen as an analogue of Hamilton's identity for Schouten solitons. These inequalities have been used in \cite{brgs} to provide control on the growth of the potential function and the volume of geodesic balls. They were also used in \cite{brgs} to prove the following growth estimates on $f$ (see also Theorem 4.2 of \cite{MuWa3}).
	
	\begin{theorem}[\cite{brgs}]\label{thm12vb}
		Let $(M^n, g, f, \lambda)$ be a complete non-compact shrinking Schouten soliton with $f$ non-constant, and let $\displaystyle f_0=\min_{M}f$ and $p_0\in M$. Then
		$$\frac{\lambda}{4}(d(p)-C_1)^2+f_0\leq f(p)\leq\lambda(d(p)+C_2)^2+f_0,$$
		where $C_1$ and $C_2$ are positive constants depending on $\lambda$ and the geometry of the soliton on the unit ball $B_{p_0}(1)$ and $d(p)=d(p,p_0)>2$.
	\end{theorem}
	
	The proof of Theorem \ref{th11val33} follows from a careful analysis of an ordinary differential inequality that we pass to describe. Let $p\in M$ be a regular point of $f$ and $\alpha_{p}:(\omega_{1}(p),\omega_{2}(p))\rightarrow M$ the maximal integral curve of $\frac{\nabla f}{\vert\nabla f\vert^2}$ through $p$. We simply write $\alpha:(\omega_{1},\omega_{2})\rightarrow M$, when the dependence on $p$ is irrelevant. The importance of these curves comes from the fact that $f\circ\alpha$ is an affine function of $s$. In fact, once $(f\circ\alpha)'(s)=1,\ \forall s\in(\omega_{1},\omega_{2})$, we readily have $(f\circ\alpha)(s_{2})-(f\circ\alpha)(s_{1})=s_{2}-s_{1}$, for any $[s_{1},s_{2}]\subset(\omega_{1},\omega_{2})$, proving the assertion. Using Proposition \ref{thmcatino}, the differential inequality is given in the following Proposition, proved in \cite{brgs}
	
	\begin{proposition}[\cite{brgs}]\label{strategylemma}
		Let $(M^n,g,f,\lambda)$, $\lambda\neq0$, be a Schouten soliton with $f$ nonconstant and $\alpha(s)$, $s\in(\omega_{1},\omega_{2})$, a maximal integral curve of $\frac{\nabla f}{\vert\nabla f\vert^2}$. The function $b:(\omega_{1},\omega_{2})\rightarrow\mathbb{R}$, defined by
		\begin{align}\label{composition}
			b(s)=\vert\nabla f(\alpha(s))\vert^2,
		\end{align}
		satisfies the differential inequality
		\begin{equation}\label{MainIneq}
			bb''-(b')^2+6\lambda b'-8\lambda^2\geq0,
		\end{equation}
		where $b'$ and $b''$ are the first and the second derivative of $b$ with respect to $s$.
	\end{proposition}
	
	The proposition above provides a way of studying $\vert\nabla f\vert^2$ along $\alpha_{p}(s)$. The proof of Theorem \ref{111valter33} in \cite{brgs} relies on estimates obtained from \eqref{MainIneq}. The most important of these estimates asserts that $2\lambda^2\leq \lambda b'(s)\leq4\lambda^2,\ \forall s\in(\omega_{1},\omega_{2})$. Following the same ideas, we obtain the following generalization of \eqref{112valter33}, which will be used in Section \ref{shrsch}.
	
	\begin{proposition}\label{newest}
		Let $(M,g,f,\lambda)$, be a gradient Schouten soliton with $f$ non-constant and such that for some constants $\alpha, \delta$ we have $0\leq\delta\lambda\leq R\lambda\leq\alpha\lambda$. Let $p$ be a regular point of $f$ $($i.e. a point such that $\nabla f(p)\neq0$$)$. Then
		\begin{equation}\label{lemma1}
			\left(\frac{\delta}{n-1}+2\lambda\right)f(p)\leq\vert\nabla f\vert^2(p) \leq\left(\frac{\alpha}{n-1}+2\lambda\right)f(p).  
		\end{equation}
	\end{proposition}
	\begin{proof}
		Let $a(s)$, $s\in(\omega_1,\omega_2)$, be a maximal integral curve of $\frac{\nabla f}{\vert\nabla f\vert^2}$, and let $b(s)=\vert\nabla f(a(s))\vert^2$.
		Notice that, since $d(\vert\nabla f\vert^2)(X)=2\nabla^2 f(X,\nabla f)$, we have from the Schouten equation (\ref{schoutensoliton}) that
		\begin{equation*}
			\begin{split}
				\mbox{Ric}(a'(s),\nabla f)+\frac{1}{2}d(\vert\nabla f\vert^2)(a'(s))=\left(\frac{R}{2(n-1)}+\lambda\right)df(a'(s)).
			\end{split}
		\end{equation*}
		Given $\mbox{Ric}(a'(s),\nabla f)=0$ (Proposition \ref{thmcatino}) and $(f\circ a)'(s)=1$, we conclude
		\begin{equation}\label{eq1}
			b'(s)=d(\vert\nabla f\vert^2)a'(s)=\frac{R}{n-1}+2\lambda.
		\end{equation}
		By integrating (\ref{eq1}) over $[s_1,s]\subset(\omega_1,\omega_2)$ we get,
		\begin{equation*}
			\lambda\int_{s_1}^s\left(\frac{R}{n-1}+2\lambda\right)=\lambda\int_{s_1}^sb'(s)=\lambda\left(b(s)-b(s_1)\right),
		\end{equation*}\label{aplicardesigualdade}
		recalling that $\lambda\delta\leq\lambda R\leq\lambda\alpha$, one has that
		\begin{equation*}
			\begin{split}
				\left(\frac{\lambda\delta}{n-1}+2\lambda^2\right)(s-s_1)\leq \lambda\left(b(s)-b(s_1)\right) \leq\left(\frac{\lambda\alpha}{n-1}+2\lambda^2\right)(s-s_1).
			\end{split}
		\end{equation*}
		It follows from $(f\circ\alpha)(s)-(f\circ\alpha)(s_{1})=s-s_{1}$ that the inequality above is equivalent to
		\begin{equation}\label{auxmaximalcurve}
			\left(\frac{\delta}{n-1}+2\lambda\right)\lambda(f(a(s))-f(a(s_1)))\leq \lambda\left(b(s)-b(s_1)\right) \leq\left(\frac{\alpha}{n-1}+2\lambda\right)\lambda(f(a(s))-f(a(s_1))).
		\end{equation}
		Let $s_0\in(\omega_1,\omega_2)$ be such that $\lim_{s\to s_0}\lambda f(a(s))= \lambda f_0$, where $\displaystyle\lambda f_0=\min_{p\in M}{\lambda f(p)}=\lambda f(p_{0})$, and, consequently, $\displaystyle \lim_{s\to s_0}b(s)=0$. Then, by doing $a(s)=f(p)$ and $s_1\to s_0$ in (\ref{auxmaximalcurve}), we have
		$$ \left(\frac{\delta}{n-1}+2\lambda\right)(f(p)-f_0)\leq\vert\nabla f\vert^2(p) \leq\left(\frac{\alpha}{n-1}+2\lambda\right)(f(p)-f_0).$$
		Noticing that, if $f$ {satisfies} the Schouten soliton, then also does $f+c$ for any constant $c$, we can assume without {loss} of generality that $f_0=0$ and the proposition is therefore proved.
	\end{proof}
	
	\subsection{Ends of Smooth Metric Measure Spaces}\label{sub2}
	We start this subsection by recalling a few definitions. An {\it end} $E$ with respect to a smooth compact subset of $M$ is an unbounded connected component of its complement, and the quantity of such connected components is the number of ends with respect to this compact subset. If this number has an upper bound, no matter how big the compact set is, we say $M$ has {\it finitely many ends} and we call the smallest upper bound $n_{0}=n_{0}(M)$ the number of ends of $M$. We also say $M$ {\it has $n_{0}$ ends}. If $n_{0}=1$, we say that $M$ is {\it connected at infinity}.

    In what follows, we consider a smooth metric measure space $(M,g,e^{-\varphi}dV)$, where $\varphi\in C^{\infty}(M)$. We say that a function $u\in C^{2}(M)$ is $\varphi$-harmonic if $\Delta_\varphi u=\Delta u-\langle \nabla \varphi,\nabla u\rangle=0$.
	
	One important tool in investigating the existence of $\varphi$-harmonic functions is Green's function for the $\varphi$-Laplacian $\Delta_\varphi$ (see \cite{pLi1} and Section 2 of \cite{wan2} for details). A Green's function is not always positive, which motivates the definitions below.
	
	A Riemannian manifold $M$ is said to be $\varphi$-{\it non-parabolic} if it admits a positive Green's function. If such a function does not exist, the manifold is called $\varphi$-{\it parabolic}. This definition can be localized to an end. An end $E$ is called $\varphi$-{\it non-parabolic} if the Laplacian, $\Delta_{\varphi}$, has a positive Green's function on $E$ satisfying the Neumann boundary condition. Otherwise, $E$ is called $\varphi$-{\it parabolic}. If $\varphi$ is constant, we just say the manifold or one of its ends is {\it parabolic} or {\it non-parabolic}, according to the appropriate Green's function. For the latter, see Chapters 17, 20 and 21 of \cite{pLi}.
	
	The concepts above are intimately related to the existence of non-negative $\varphi$-subharmonic functions. Recall that $u$ is $\varphi$-subharmonic if $\Delta_\varphi u\geq0$. It is $\varphi$-superharmonic if $-u$ is $\varphi$-subharmonic. The following criterion for detecting $\varphi$-parabolicity will be used in this paper.
	
	\begin{proposition}\label{imptoo1}
		Let $E$ be an end of a smooth metric measure space $(M,g,e^{-\varphi}dV)$. If there exists a positive $\varphi$-superharmonic function $u$ defined on $E$ satisfying $\displaystyle\liminf_{x\rightarrow E(\infty)}u(x)=0$, then $E$ is $\varphi$-non-parabolic.
	\end{proposition}
	
	The result above can be proved using the same arguments employed in the proof of Lemma 2.3 of \cite{wan2}. Compare the result above with Theorem 17.3 of \cite{pLi}.
	
	From the theory developed by Li and Tam in \cite{pLi0}, the dimension of certain spaces of harmonic {functions} helps estimating the number of ends. One aspect in this theory is that the existence of more than one end provides the existence of nonconstant harmonic functions satisfying special properties. This was first developed by taking into account the notion of non-parabolicity. However, as observed by Munteanu and Wang \cite{MuWa1,MuWa2,MuWa4,MuWa5}, this also holds true when considering $\varphi$-non-parabolicity. When the manifold is $\varphi$-non-parabolic, the following result is an important tool used in \cite{MuWa1,MuWa2,MuWa4,MuWa5}, and will be required in the next sections.
	
	\begin{proposition}\label{imptoo2}
		If the smooth metric measure space $(M,g,e^{-\varphi}dV)$ has at least two $\varphi$-non-parabolic ends, then there is a nonconstant bounded $\varphi$-harmonic function $u$ so that
		\begin{align*}	
			\int_{M}\vert\nabla u\vert e^{-\varphi}<\infty.
		\end{align*}
	\end{proposition}
	
	When $\varphi$ is constant, the corresponding result is due to Li and Tam, and was obtained in \cite{pLi0}. See also \cite{endharm} and Theorem 21.3 of \cite{pLi}.
	
	\section{Ends of Shrinking Schouten Soliton}\label{shrsch}
	
	The purpose of this section is to prove Theorem \ref{thm1item1} and Theorem \ref{thm1}. 
	
	\begin{proof}[{\bf Proof of Theorem \ref{thm1item1}}]
		Consider real numbers satisfying $f_{0}<a<b$, where $f_{0}$ is the minimum value of $f$ given in Theorem \ref{th11val33}, and define the sets $M_{a}^{b}=\{p\in M;a\leq f(p)\leq b\}$ and $M^{b}=\{p\in M;f(p)\leq b\}$.  from Theorem \ref{thm12vb} $f$ is proper and, consequently, $M_{a}^{b}$ is compact. On the other hand, using Theorem \ref{th11val33} and Theorem \ref{thm12vb} we conclude that for any $p\in M$
				\begin{align*}
					\vert\nabla f\vert^2(p)\geq2\lambda(f(p)-f_0)\geq\frac{\lambda^2}{2}(d(p)-C_1)^2.
				\end{align*}
		This means that we can choose $a>f_{0}$ so that $f$ has no critical points in $M_{a}^{b}$, for any $b>a$. Observing that $M^{b}=M^{a}\cup M_{a}^{b}$ and that $ M_{a}^{b}$ is compact, Morse theory (Theorem 3.1 of \cite{mil}) assures $M^{a}$ is diffeomorphic to $M^{b}$. Writing $\displaystyle M=\cup_{a\leq b}M^{b}$, we see that $M$ is diffeomorphic to the interior of a compact manifold with boundary, i.e., has finite topological type.
		
		As a consequence of the discussion above, $\partial M^{a}$ and $\partial M^{b}$ are diffeomorphic. Hence, they have the same quantity of boundary components for any $b>a$, showing $M$ has finitely many ends. This also shows each end is diffeomorphic to a component of $\partial M^{a}$ times $[0,\infty)$.
	\end{proof}

	\begin{proof}[{\bf Proof of Theorem \ref{thm1} item i)}]
		Suppose $M$ has two $f$-non-parabolic ends. By Proposition \ref{imptoo2}, $M$ admits a positive non-constant bounded $f$-harmonic function $v$ such that
		$$\int\limits_M\vert\nabla v\vert^2e^{-f}<\infty.$$
		Recalling $\mbox{Ric}_f=\mbox{Ric}+\nabla^2 f$ we get from the Schouten soliton equation that
		$$\mbox{Ric}_f=\left(\frac{R}{2(n-1)}+\lambda\right)g.$$
		Thus, by the Bochner formula for the weighted Laplacian we conclude
		\begin{equation*}
			\begin{split}
				\frac{1}{2}\Delta_f \vert\nabla v\vert^2=&\vert\nabla^2 v\vert^2+\langle\nabla\Delta_f v, \nabla v \rangle+\mbox{Ric}_f(\nabla v,\nabla v)\\
				=&\vert\nabla^2 v\vert^2 +\left(\frac{R}{2(n-1)}+\lambda\right)\vert\nabla v\vert^2.
			\end{split}
		\end{equation*}
		Let $\phi$ be a cut-off function over $M$ such that $\phi=1$ on $B_p(r)$ and $\phi=0$ outside $B_p(2r)$, for $p\in M$ and $r>0$. Then, integrating over $M$ we get from the identity above, the divergence theorem and the Schwarz and Young inequalities that
		\begin{equation*}
			\begin{split}
				2\int\limits_M\vert\nabla^2 v\vert^2\phi^2\mbox{e}^{-f} +\int\limits_M\left(\frac{R}{(n-1)}+2\lambda\right)\vert\nabla v\vert^2\phi^2\mbox{e}^{-f}=&-\int\limits_M\langle\nabla\vert\nabla v\vert^2,\nabla\phi^2\rangle\mbox{e}^{-f}\\
				&=-4\int\limits_M\phi\nabla^2 v\langle\nabla v,\nabla\phi\rangle\mbox{e}^{-f}\\
				&\leq 4\int\limits_M\vert\phi\vert\vert\nabla^2 v\vert\vert\nabla v\vert\vert\nabla\phi\vert\mbox{e}^{-f}\\
				&\leq 2\int\limits_M \vert\nabla^2 v\vert^2\phi^2\mbox{e}^{-f} + 2\int\limits_M\vert\nabla v\vert^2\vert\nabla\phi\vert^2\mbox{e}^{-f}.
			\end{split}
		\end{equation*}
		Once $R\geq0$, the expression above implies
		$$\lambda\int\limits_M\vert\nabla v\vert^2\phi^2\mbox{e}^{-f}\leq\int\limits_M\vert\nabla v\vert^2\vert\nabla\phi\vert^2\mbox{e}^{-f}.$$
		Since $\int_M\vert\nabla v\vert^2e^{-f}<\infty$, the right side must tend to zero as $r\to\infty$ (because $\vert\nabla\phi\vert\to 0$), this forces $\vert\nabla v\vert^2=0$ as $\lambda$ is not null. Therefore $v$ must be constant, which is a contradiction. Thus $M$ must have at most one $f$-non-parabolic end, as we wanted to prove.
	\end{proof}

	\begin{proof}[{\bf Proof of Theorem \ref{thm1} item ii)}]
		Taking the trace over (\ref{schoutensoliton}) we know that
		$$\Delta f=\left(\frac{n}{2(n-1)}-1\right)R+n\lambda.$$
		Since $R\leq\alpha$, we have
		$$\Delta f\geq n\lambda-\frac{n-2}{2(n-1)}\alpha.$$
		Thus, by taking
		$$a=\frac{2(n-1)(n-2)\lambda-n\alpha}{2(\alpha+2\lambda)}>0,$$
		we get from (\ref{lemma1}) that
		\begin{equation}
			\begin{split}
				\Delta f^{-a}&=-af^{-a-1}\Delta f+a(a+1)\vert\nabla f\vert^2f^{-a-2}\\
				&\leq -af^{-a-1}\left( n\lambda-\frac{n-2}{2(n-1)}\alpha \right)+a(a+1)\left(\frac{\alpha}{n-1}+2\lambda\right)f^{-a-1}\\
				&=af^{-a-1}\left[(a+1)\left(\frac{\alpha}{n-1}+2\lambda\right)-\left( n\lambda-\frac{n-2}{2(n-1)}\alpha \right)\right]=0.
			\end{split}
		\end{equation}
		On the other hand, from Theorem \ref{thm12vb} there is a constant $C>0$ such that $C(d(p))^2<f(p)$, where $d(p)$ is the distance function from a fixed point $p_0\in M$, and then
		$$f^{-a}\to0,\ \  \mbox{ as }\ \  {p\to+\infty}.$$
		
		Thus, the restriction of $f^{-a}$ to an end $E$ of $M$ provides a positive super-harmonic function which converges to zero at infinity and, by Proposition \ref{imptoo1}, $E$ is non-parabolic.
	\end{proof}
	
	\section{Volume and Ends of Expanding Schouten Solitons}\label{expsch}
	
	In this section, we prove the results concerning gradient expanding Schouten solitons. We first establish a quadratic lower decay for $f$. This is fundamental in the proof of Theorem \ref{grwth}, which is presented in the sequence.
	
	\begin{proposition}
		Let $(M,g,f,\lambda)$ be a complete non-trivial expanding Schouten soliton, $p_{0}$ a maximum of $f$, i.e., $f_{0}=f(p_{0})$, and $p\in M$ a fixed point.Then
		\begin{align}\label{potgrwint}
			f(x)\geq\lambda(r(x))^2-2\sqrt{-\lambda}\sqrt{f_{0}-f(p)}r(x)+f(p)
		\end{align}
		where $r(x)=d(x,p)$. If we take $p=p_{0}$, then
		\begin{align}\label{potgrw}
			f(x)\geq\lambda(r(x))^2+f_{0},
		\end{align}
	\end{proposition}
	\begin{proof}
		Fix a point $p\in M$. It follows from \eqref{112valter33} that 
		\begin{align}\label{lipineq}
			\vert\nabla\sqrt{f_{0}-f}\vert\leq\sqrt{-\lambda}.
		\end{align}
		Consequently, $\sqrt{f_{0}-f}$ is a Lipschitzian. Integrating the right hand side of \eqref{lipineq} along minimal geodesics emanating from $p$ gives
		\begin{align*}
			\sqrt{f_{0}-f(x)}-\sqrt{f_{0}-f(p)}\leq\sqrt{-\lambda}r(x),
		\end{align*}
		which proves \eqref{potgrwint}.
	\end{proof}

	For future use we remark that from \eqref{112valter33} and \eqref{potgrwint} there exist constants $c_{1}>0$ and $c_{2}>0$ so that
	\begin{align}\label{gradest}
		\vert\nabla f\vert^2
		&\leq4\lambda^2r^2+c_{1}r+c_{2}.
	\end{align}

	\begin{proof}[{\bf Proof of Theorem \ref{grwth}}]
		Our proof follows the ideas established in the proofs of Theorem 5.1 of \cite{MuWa2} and Theorem 5.2 of \cite{wan1}. 
		
		In what follows we fix $p\in M$. First notice that the lower bound of \eqref{grtbhv} follows immediately from \eqref{potgrwint}. Before establishing the upper bound of \eqref{grtbhv}, we will prove the upper bound on the volume growth stated in \eqref{volgrwt}. 
		
		To estimate the volume, first recall that on polar coordinates one has $dV\large|_{exp_p(r\xi)}=J(x,r,\xi)drd\xi$. In what follows we will write $J(r)$ instead of $J(x,r,\xi)$, and $J'(r)$ for the derivative of $J(r)$ with respect to $r$. Once $Ric_{f}\geq2\lambda g$, it follows from the Bochner formula for the $f$-Laplacian that
		\begin{align*}
			\left(\frac{J'}{J}\right)'+\frac{1}{n-1}\left(\frac{J'}{J}\right)^2+2\lambda-f''\leq0.
		\end{align*}
		Integrating the inequality above we conclude the function $u(r)=\frac{J'(r)}{J(r)}$ satisfies
		\begin{align}
			u(r)+\frac{1}{(n-1)r}\left(\int_{1}^{r}u(s)ds\right)^2\leq-4\lambda r+c_{4},
		\end{align}
		where we have used $f'(r)\leq-2\lambda r+\tilde{c}$, which follows from \eqref{gradest}, for some constant $\tilde{c}>0$. Using the definition of $u(r)$, we conclude that
		\begin{align*}
			\ln{J(r)}-\ln{J(1)}=\int_{1}^{r}u(s)ds\leq\sqrt{-4(n-1)\lambda}r+\sqrt{-\frac{n-1}{4\lambda}}c_{4}.
		\end{align*}
		Integrating the inequality above on $\xi\in\mathbb{S}^{n-1}\subset T_{p}M$ for a fixed $r$, we obtain the area of $\partial B_{p}(r)$
		\begin{align}\label{area}
			A(\partial B_{p}(r))\leq Ce^{\sqrt{-4(n-1)\lambda}r}.
		\end{align}
		Integrating further with respect to $r$ we obtain \eqref{volgrwt}.
		
		Now we prove the upper bound of \eqref{grtbhv}. In order to do that, fix $\varepsilon$ in the interval $(0,1)$ and define $u=f_{0}-f+\varepsilon$. It follows from Theorem \ref{th11val33} that
		\begin{align}
			&u\geq\varepsilon>0\label{relabled0},\\
			&-2\lambda(u-\varepsilon)\leq\vert\nabla u\vert^2\leq-4\lambda(u-\varepsilon),\label{relabled}\\
			&\Delta u=\frac{n-2}{2(n-1)}-n\lambda.\label{relabled2}
		\end{align}
		For a given $k$, a straightforward computation gives
		\begin{align}\label{comp1}
			\Delta e^{2k\sqrt{u}}=\left(\frac{k}{\sqrt{u}}\left(\frac{n-2}{2(n-1)}R-n\lambda\right)+\left(\frac{k^2}{u}-\frac{k}{2u\sqrt{u}}\right)\vert\nabla u\vert^2\right)e^{2k\sqrt{u}}.
		\end{align}
		On the other hand, using Green's identity for $\Delta$ and integrating on the geodesic ball $B_{p}(r)$ one has
		\begin{align}
			\begin{split}\label{comp2}
				\int_{B_{p}(r)}(\Delta e^{2k\sqrt{u}})u^{k^{2}}&=k\int_{\partial B_{p}(r)}\frac{1}{\sqrt{u}}\frac{\partial u}{\partial r}u^{k^{2}}e^{2k\sqrt{u}}-k^3\int_{B_{p}(r)}\frac{\vert\nabla u\vert^2}{\sqrt{u}}u^{k^{2}-1}e^{2k\sqrt{u}}\\
				&\leq k\int_{\partial B_{p}(r)}\frac{\vert\nabla u\vert}{\sqrt{u}}u^{k^{2}}e^{2k\sqrt{u}}-k^3\int_{B_{p}(r)}\frac{\vert\nabla u\vert^2}{\sqrt{u}}u^{k^{2}-1}e^{2k\sqrt{u}}.
			\end{split}
		\end{align}
		Inserting \eqref{comp1} into \eqref{comp2} and reordering terms we obtain
		\begin{align*}
			\int_{\partial B_{p}(r)}\frac{\vert\nabla u\vert}{\sqrt{u}}u^{k^{2}}e^{2k\sqrt{u}}\geq\int_{B_{p}(r)}\left(\frac{1}{\sqrt{u}}\left(\frac{n-2}{2(n-1)}R-n\lambda\right)+\left(\frac{k}{u}+\frac{2k^2-1}{2u\sqrt{u}}\right)\vert\nabla u\vert^2\right)u^{k^{2}}e^{2k\sqrt{u}}
		\end{align*}
		If we assume $2k^2>1$, the inequality above, together with \eqref{relabled2} and the left hand side of \eqref{relabled}, implies that
		\begin{align}\label{intinterm}
			\int_{\partial B_{p}(r)}\frac{\vert\nabla u\vert}{\sqrt{u}}u^{k^{2}}e^{2k\sqrt{u}}\geq-2\lambda\int_{B_{p}(r)}\left(\frac{1}{\sqrt{u}}+\frac{k(u-\varepsilon)}{u}+\frac{(2k^2-1)(u-\varepsilon)}{2u\sqrt{u}}\right)u^{k^{2}}e^{2k\sqrt{u}}.
		\end{align}
		
		In what follows we will prove that for a fixed $x$ the following inequality holds
		\begin{align}\label{ineq1}
			2\sqrt{-\lambda}\left(\frac{1}{\sqrt{u}}+\frac{k(u-\varepsilon)}{u}+\frac{(2k^2-1)(u-\varepsilon)}{2u\sqrt{u}}\right)\geq k \frac{\sqrt{-4\lambda(u-\varepsilon)}}{\sqrt{u}}.
		\end{align}
		First rewrite it as 
		\begin{align*}
			1+k\sqrt{u}-k\varepsilon\frac{\sqrt{u}}{u}+\left(k^2-\frac{1}{2}\right)\left(1-\frac{\varepsilon}{u}\right)\geq k\sqrt{u-\varepsilon},
		\end{align*}
		define $\gamma>0$ by the equality $\varepsilon\gamma^2=u$ and observe that \eqref{relabled0} implies $\gamma\geq1$. Rewriting the inequality above in terms of $\gamma$ we get
		\begin{align*}
			\left(1-\frac{1}{\gamma}\right)k^2+\left(\frac{\gamma^2-1-\gamma\sqrt{\gamma^2-1}}{\gamma}\right)\sqrt{\varepsilon}k+\frac{\gamma^2+1}{2\gamma^2}\geq0.
		\end{align*}
		If $\gamma=1$, the inequality above is clearly true. Let us assume $\gamma>1$. In this case we define $p(x)=Ax^2+Bx+C$, for $A=1-\frac{1}{\gamma}$, $B=\left(\frac{\gamma^2-1-\gamma\sqrt{\gamma^2-1}}{\gamma}\right)\sqrt{\varepsilon}$ and $C=\frac{\gamma^2+1}{2\gamma^2}$. With this notation, \eqref{ineq1} is equivalent to $p(k)\geq0$, for any $k\geq\frac{1}{2}$. The last assertion will be proved if we show that the discriminant of $p(x)$ satisfies $B^2-4AC\leq0$. A simple calculation gives
		\begin{align*}
			B^2-4AC&=\left(\frac{\gamma^2-1}{\gamma^2}\right)\left((2\gamma^2-1-2\gamma\sqrt{\gamma^2-1})\varepsilon-\frac{2(\gamma^2+1)}{\gamma^2}\right).
		\end{align*}		
Since $0<2\gamma^2-1-2\gamma\sqrt{\gamma^2-1}<1$ one has
		\begin{align*}
			B^2-4AC&<\left(\frac{\gamma^2-1}{\gamma^2}\right)\left(\varepsilon-2\right).
		\end{align*}
		As we have taken $\varepsilon\in(0,1)$, we conclude that $B^2-4AC<0$, and \eqref{ineq1} follows.
		
		Substituting \eqref{ineq1} in \eqref{intinterm} and using the right hand side of \eqref{relabled}, we conclude that
		\begin{align}\label{impineq}
			\int_{\partial B_{p}(r)}\vert\nabla u\vert u^{k^{2}-\frac{1}{2}}e^{2k\sqrt{u}}\geq \sqrt{-\lambda}k\int_{B_{p}(r)}\vert\nabla u\vert u^{k^{2}-\frac{1}{2}}e^{2k\sqrt{u}}.
		\end{align}
		Setting $\varphi(r)=\int_{B_{p}(r)}\vert\nabla u\vert u^{k^{2}-\frac{1}{2}}e^{2k\sqrt{u}}$ and using the Coarea Formula, the inequality above can be rewritten as $\varphi'(r)\geq \sqrt{-\lambda}k\varphi(r)$, for $r\geq0$. As $f$ is non-constant, there is $r_{0}>0$ so that $\varphi(r_{0})>0$. Integrating the differential inequality we conclude the existence of a positive constant $c>0$ so that $\varphi(r)\geq ce^{\sqrt{-\lambda}kr}$, $\forall r\geq r_{0}$, which from \eqref{impineq} implies
		\begin{align*}
			\int_{\partial B_{p}(r)}\vert\nabla u\vert u^{k^{2}-\frac{1}{2}}e^{2k\sqrt{u}}\geq c e^{\sqrt{-\lambda}kr}.
		\end{align*}
		Maximizing $u$ on $\partial B_{p}(r)$ in the inequality above, using \eqref{gradest} and \eqref{area} we conclude that
		\begin{align}
			\sup_{\partial B_{p}(r)}{e^{2k\sqrt{u}}}\geq Ce^{\sqrt{-\lambda}kr-2k^2\ln r-\sqrt{-4(n-1)\lambda}r},
		\end{align}
		where we have taken $r>>1$. As a consequence
		\begin{align*}
			\sup_{\partial B_{p}(r)}{{\sqrt{u}}}\geq \frac{\sqrt{-\lambda}}{2}r-k\ln r-c\frac{r}{k},
		\end{align*}
		for a positive constant $c$. It is not hard to see that for a fixed $r$, the maximum of the expression on the right hand side of the inequality above is assumed at $k=\sqrt{\frac{cr}{\ln r}}$. As a consequence, we obtain
		\begin{align*}
			\sup_{\partial B_{p}(r)}{{(f_{0}-f+\varepsilon})}\geq\left(\frac{\sqrt{-\lambda}}{2}r-\sqrt{cr\ln r}\right)^2\geq-\frac{\lambda}{4}r^2-cr^{\frac{3}{2}}\sqrt{\ln r}.
		\end{align*}
		The desired upper bound follows letting $\varepsilon$ to $0$.
	\end{proof}
	
	For the proof of Theorem \ref{thexpss}, we need the following auxiliary result, which provides a weighted Poincaré inequality and a positive lower bound to $\mu_{1}(\Delta_{f})$, the spectrum of $\Delta_{f}$.
	\begin{proposition}\label{lem51forsxs}
		Let $(M,g,f,\lambda)$ be a complete non-trivial expanding Schouten soliton and set
		\begin{align*}
			\sigma=\frac{n-2}{2(n-1)}R-n\lambda.
		\end{align*}
		Then $\sigma\geq-2\lambda>0$ and
		\begin{align}\label{weipoicineq}
			\int\limits_M \sigma\phi^2\mbox{e}^{-f}\leq\int\limits_M\vert\nabla\phi\vert^2\mbox{e}^{-f},
		\end{align}
		for any $\phi\in C_0^\infty(M)$. In particular, $\mu_{1}(\Delta_{f})\geq-2\lambda$ and $M$ is $f$-non-parabolic.
	\end{proposition}
	\begin{proof}
		From Theorem \ref{th11val33}, we know that
		$2(n-1)\lambda\leq R\leq 0$, then, 
		\begin{align}\label{sign}
			\sigma\geq\frac{2(n-2)(n-1)\lambda}{2(n-1)}-n\lambda=-2\lambda>0.
		\end{align}
		By using the trace of equation (\ref{schoutensoliton}) we have
		\begin{align}\label{suphar}
			\Delta_f(\mbox{e}^f) =\left( \Delta_f(f)+\vert\nabla f\vert^2 \right)\mbox{e}^f=( \Delta f)\mbox{e}^f=-\left(\frac{n-2}{2(n-1)}R -n \lambda\right)\mbox{e}^f=-\sigma\mbox{e}^f
		\end{align}
		and, using Proposition 1.1 of \cite{pLiW3}, we obtain \eqref{weipoicineq}.

		Now, it follows from Corollary 1.4 of \cite{pLiW3} that the weighted Poincaré inequality proved above characterizes $M$ as an $f$-non-parabolic manifold.
		
		Finally, using \eqref{weipoicineq} and the lower bound of $\sigma$, we conclude that
		\begin{align*}
			-2\lambda\leq\frac{\int\limits_M\vert\nabla\phi\vert^2\mbox{e}^{-f}}{\int\limits_M\phi^2\mbox{e}^{-f}},
		\end{align*}
		for any $\phi\in C_0^\infty(M)$. Using the variational characterization of the spectrum of $\Delta_{f}$ we conclude that $\mu_{1}(\Delta_{f})\geq-2\lambda$.
	\end{proof}

	Now we proceed to the proof of Theorem \ref{thexpss}. First we show that $M$ has only $f$-non-parabolic ends. In the second step, we prove that there can be only one $f$-non-parabolic end in $M$, concluding $M$ must be connected at infinity or rigid.
	
	\begin{proof}[{\bf Proof of Theorem \ref{thexpss}}]
		From Proposition \ref{lem51forsxs} we know that $M$ is $f$-non-parabolic. We will show all ends of $M$ are $f$-non-parabolic. In order to do that, notice that by \eqref{sign} and \eqref{suphar}, $e^{f}$ is positive and $f$-superharmonic. If $E$ is an end of $M$, it follows from the growth behavior of $f$ given in \eqref{grtbhv} that $\displaystyle\liminf_{x\in E}e^{f(x)}=0$. By Proposition \ref{imptoo1}, $E$ is $f$-non-parabolic.
		
		Now, suppose $M$ is not connected at infinity. We will show $M$ is isometric to $\mathbb{R}\times N^{n-1}$, with $N$ compact. Under this assumption, by Proposition \ref{imptoo2}, there exists a positive non-constant $f$-harmonic function $h$ such that $h<1$ and
		\begin{equation}\label{energyh}
			\int\limits_M\vert\nabla h\vert^2\mbox{e}^{-f}<\infty.
		\end{equation}
		Once $\Delta_fh=0$, the Bochner formula for $\Delta_{f}$, equation (\ref{schoutensoliton}) and Kato's inequality give
		\begin{equation*}
			\begin{split}
				\frac{1}{2}\Delta_f\vert\nabla h\vert^2 =&\vert\nabla^2h\vert^2 +\mbox{Ric}_f(\nabla h,\nabla h)\\
				\geq&\vert\nabla\vert\nabla h\vert\vert^2+\left(\frac{R}{2(n-1)}+\lambda\right)\vert\nabla h\vert^2,
			\end{split}
		\end{equation*}
		and, since $\frac{1}{2}\Delta_f\vert\nabla h\vert^2=\vert\nabla h\vert\Delta_f\vert\nabla h\vert+\vert\nabla\vert\nabla h\vert\vert^2$, we obtain
		\begin{equation}\label{ineq1thexpss}
			\Delta_f\vert\nabla h\vert\geq\left(\frac{R}{2(n-1)}+\lambda\right)\vert\nabla h\vert.
		\end{equation}
		Consider a cut-off function $\phi$ and insert $\vert\nabla h\vert\phi$ in Poincaré's inequality \eqref{weipoicineq} to get
		\begin{equation}\label{ineq2thexpss}
			\begin{split}
				\int\limits_M \sigma\vert\nabla h\vert^2\phi^2\mbox{e}^{-f} \leq\int\limits_M\left(\vert\nabla\vert\nabla h\vert\vert^2\phi^2+2\vert\nabla h\vert\phi\langle\nabla\vert\nabla h\vert,\nabla\phi\rangle+\vert\nabla h\vert^2\vert\nabla\phi\vert^2\right)\mbox{e}^{-f}.
			\end{split}
		\end{equation}
		On the other hand, from Green's identity for $\Delta_{f}$ we know that
		\begin{equation}\label{ineq3thexpss}
			\begin{split}
				-\int\limits_M\vert\nabla h\vert(\Delta_f\vert\nabla h\vert)\phi^2\mbox{e}^{-f}=\int\limits_M\langle\nabla\vert\nabla h\vert,\nabla(\vert\nabla h\vert\phi^2)\rangle\mbox{e}^{-f}=\int\limits_M\langle\nabla\vert\nabla h\vert,\nabla\vert\nabla h\vert\phi^2+2\phi\vert\nabla h\vert\nabla\phi\rangle\mbox{e}^{-f}.
			\end{split}
		\end{equation}
		Thus, combining \eqref{ineq2thexpss} with \eqref{ineq1thexpss} and \eqref{ineq3thexpss} and regrouping terms we conclude that
		\begin{equation*}
			\begin{split}
				\int\limits_M\left(\sigma+\frac{R}{2(n-1)}+\lambda\right)\vert\nabla h\vert^2\phi^2\mbox{e}^{-f}\leq\int\limits_M\vert\nabla h\vert^2\vert\nabla\phi\vert^2\mbox{e}^{-f}
			\end{split}
		\end{equation*}
		Since $\vert\nabla h\vert$ has finite energy, we can pick a sequence of cut-off functions $\phi_{R}$ with $\phi_{R}\equiv1$ in $B_{p}(R)$ and support in $B_{p}(2R)$. For such functions, we obtain $\int_M\vert\nabla h\vert^2\vert\nabla\phi\vert\mbox{e}^{-f}\to 0$, as $R\rightarrow\infty$.
		
		Recalling $\sigma=\frac{n-2}{2(n-1)}R-n\lambda$, the inequality above implies
		\begin{equation*}
			\begin{split}
				0\leq\int\limits_M \left(\frac{1}{2}R-(n-1)\lambda\right)\vert\nabla h\vert^2\mbox{e}^{-f}\leq0,
			\end{split}
		\end{equation*}
		where the inequality on the left follows from (\ref{111valter33}). Once the integral above vanishes, the integrand is non-negative and $h$ is non-constant, we must have $R=2(n-1)\lambda$, and \eqref{schoutensoliton} turns into
		\begin{align*}
			\mbox{Ric}+\nabla^2f=2\lambda g.
		\end{align*}
		As a consequence, $(M,g,f,\lambda_{0})$ is a gradient Ricci soliton with constant scalar curvature  $R=(n-1)\lambda_{0}$, with $\lambda_0=2\lambda$. It thus follows from Theorem \ref{thRNcompact} that $M$ must be isometric to the product $\mathbb{R}\times N^{n-1}$.
	\end{proof}


\begin{thebibliography}{0}
					\bibitem{brgs} V. Borges. \textit{On complete gradient Schouten solitons}. Nonlinear Anal. 221 (2022), Paper No. 112883, 15 pp.
					
					
					\bibitem{brstrs} V. Borges, H. A. Rosero-Garc\'ia and J. P. dos Santos \textit{Ends of shrinking gradient $\rho $-Einstein solitons}. arXiv preprint (2023) arXiv:2308.07182.
					
					\bibitem{brnt} R. L. Bryant. \textit{Ricci flow solitons in dimension three with $SO(3)$-symmetries}. preprint. (2005). 
					
					\bibitem{bourg} J. Bourguignon \textit{Ricci curvature and Einstein metrics}. Global differential geometry and global analysis (Berlin, 1979), pp. 42--63, Lecture Notes in Math., 838, Springer, Berlin, 1981.

                    \bibitem{caochenzhou} H. D. Cao, B. L. Chen and X. P. Zhu.\textit{ Recent developments on the Hamilton’s Ricci Flow}. Surveys in differential geometry, 12(1), pp.47-112, 2007.
            
					
					\bibitem{catino2} G. Catino, L. Cremaschi, Z. Djadli, C. Mantegazza and L. Mazzieri. \textit{The Ricci–Bourguignon flow}. Pacific J. Math. 287 (2017), no. 2, 337--370.
					
					\bibitem{catino} G. Catino and L. Mazzieri. \textit{Gradient Einstein solitons}. Nonlinear Anal. 132 (2016), 66--94.
					
					\bibitem{catino1} G. Catino, L. Mazzieri and S. Mongodi. \textit{Rigidity of gradient Einstein shrinkers}. Commun. Contemp. Math. 17 (2015), no. 6, 1550046, 18 pp.
					
					
					
					
					\bibitem{famazha} F.-q. Fang, J.-w. Man, Z.-l. Zhang. \textit{Complete gradient shrinking Ricci solitons have finite topological type}. C. R. Math. Acad. Sci. Paris 346 (2008), no. 11-12, 653--656.
					
					
					
					\bibitem{dema} A. Derdzinski and G. Maschler. \textit{Local classification of conformally-Einstein Kähler metrics in higher dimensions}. Proc. London Math. Soc. (3) 87 (2003), no. 3, 779--819.
					
					
					\bibitem{hamilton1} R. S. Hamilton. \textit{Three-manifolds with positive Ricci curvature}. J. Differential Geometry 17 (1982), no. 2, 255--306.
					
					
					\bibitem{lelistop} C. D. Lellis and P. M. Topping. \textit{Almost-Schur lemma}. Calc. Var. Partial Differential Equations 43 (2012), no. 3-4, 347--354.
					
					
					\bibitem{pLi} P. Li. \textit{Geometric analysis}. Cambridge Studies in Advanced Mathematics, 134. Cambridge University Press, Cambridge, 2012. x+406 pp. ISBN: 978-1-107-02064-1
					
					\bibitem{pLiW} P. Li and J. Wang. \textit{Complete manifolds with positive spectrum}. J. Differential Geom. 58 (2001), no. 3, 501--534.
					
					\bibitem{pLiW2} P. Li and J. Wang. \textit{Complete manifolds with positive spectrum, II}. J. Differential Geom. 62 (2002), no. 1, 143--162.
					
					\bibitem{pLiW3} P. Li and J. Wang. \textit{Weighted Poincaré inequality and rigidity of complete manifolds}. Ann. Sci. École Norm. Sup. (4) 39 (2006), no. 6, 921--982.
					
					\bibitem{pLi0} P. Li, L. F. Tam \textit{Harmonic functions and the structure of complete manifolds}. J. Differential Geom. 35 (1992), no. 2, 359--383.
					
					\bibitem{pLi1} P. Li, L. F. Tam \textit{Symmetric Green's functions on complete manifolds}. Amer. J. Math. 109 (1987), no. 6, 1129--1154.
					
					
					\bibitem{ma} G. Maschler. \textit{Almost soliton duality}. Adv. Geom. 15 (2015), no. 2, 159--166.
					
					
					\bibitem{mil} J. W. Milnor, \textit{Morse theory}. Based on lecture notes by M. Spivak and R. Wells. Annals of Mathematics Studies, No. 51. Princeton University Press, Princeton, NJ, 1963. {\rm vi}+153 pp.
					
					\bibitem{MuSe} O. Munteanu, N. Sesum. \textit{On gradient Ricci solitons}. J. Geom. Anal. 23 (2013), no. 2, 539--561.
					
					\bibitem{MuWa1} O. Munteanu, J. Wang. \textit{Smooth metric measure spaces with non-negative curvature}. Comm. Anal. Geom. 19 (2011), no. 3, 451--486.
					
					\bibitem{MuWa2} O. Munteanu, J. Wang. \textit{Analysis of weighted Laplacian and applications to Ricci solitons}. Comm. Anal. Geom. 20 (2012), no. 1, 55--94.
					
					\bibitem{MuWa3} O. Munteanu, J. Wang. \textit{Geometry of manifolds with densities}. Adv. Math. 259 (2014), 269--305.
					
					\bibitem{MuWa4} O. Munteanu, J. Wang. \textit{Topology of Kähler Ricci solitons}. J. Differential Geom. 100 (2015), no. 1, 109--128.
					
					\bibitem{MuWa5} O. Munteanu, J. Wang. \textit{Ends of Gradient Ricci Solitons}. The J. Geom. Anal. 32 (2022), no. 12, Paper No. 303, 26 pp.
					
					
					
					\bibitem{tobabame} W. Tokura, M. Barboza, E. Batista, I. Menezes. \textit{Rigidity Results for Riemann and Schouten Solitons}. Mediterr. J. Math. 20 (2023), no. 3, Paper No. 112, 9 pp.
					
					
					\bibitem{endharm} C. J. Sung, L. F. Tam, J. Wang \textit{Spaces of harmonic functions}. J. London Math. Soc. (2) 61 (2000), no. 3, 789--806.
					
					
					\bibitem{wan1} L. F. Wang. \textit{On Ricci-harmonic metrics}. Ann. Acad. Sci. Fenn. Math. 41 (2016), no. 1, 417--437.
					
					\bibitem{wan2} L. F. Wang. \textit{On $f$-non-parabolic ends for Ricci-harmonic metrics}. Ann. Global Anal. Geom. 51 (2017), no. 1, 91--107.
					
					\bibitem{weiwilie} G. Wei and W. Wylie. \textit{Comparison geometry for the Bakry-Emery Ricci tensor}. J. Differential Geom. 83 (2009), no. 2, 377--405.
				\end{thebibliography}
			\end{document}